\documentclass[letterpaper, 10 pt, conference]{ieeeconf}  

\IEEEoverridecommandlockouts                              
\newcommand{\Int}{\operatorname{{\mathrm int}}}
\newcommand{\In}{\operatorname{{\mathrm In}}}

\usepackage[utf8]{inputenc}

\usepackage{hyperref}       
\hypersetup{colorlinks=true, linkcolor=blue, breaklinks=true, urlcolor=blue}

\usepackage{amsmath,amsthm,amssymb,mathrsfs,mathtools,verbatim}
\newcommand{\e}      {{\varepsilon}}
\newcommand{\R}      {{\mathbb{R}}}
\renewcommand{\o}      {{\omega}}
\newcommand{\T}      {{\mathsf{T}}}
\renewcommand{\d}      {{\delta}}

\newcommand*\diff{\mathop{}\!\mathrm{d}}

\newcommand{\be}{\begin{equation}}
\newcommand{\ee}{\end{equation}}

\newtheorem{theorem}{Theorem} 
\newtheorem{Theorem}[theorem]{Theorem}
\newtheorem{Assumption}[theorem]{Assumption}

\newtheorem{Remark}[theorem]{Remark}

\newtheorem{Corollary}[theorem]{Corollary}
\newtheorem{Example}[theorem]{Example}

\newcommand{\sbm}[1]{\left[\begin{smallmatrix} #1
\end{smallmatrix}\right]}
\newcommand{\bbm}[1]{\left[\begin{matrix} #1 \end{matrix}\right]}

\title{\LARGE \bf
On singularly perturbed systems that are monotone with respect to a matrix cone of rank~$k$\thanks{This work was partially supported by a research grant from  the ISF.}}

\author{Ron Ofir,  Pietro Lorenzetti, 
and Michael Margaliot
\thanks{RO and PL contributed equally to this work. 
PL and MM   are with the School of Electrical Engineering, 
Tel Aviv University, Ramat Aviv, Israel  69978. RO is with
	the Andrew and Erna Viterbi Faculty of Electrical and Computer Engineering,
	Technion--Israel Institute of Technology, Haifa 3200003, Israel.
  Correspondence  {\tt\small  michaelm@tauex.tau.ac.il}}%
}

\begin{document}

\maketitle
\thispagestyle{empty}
\pagestyle{empty}

\begin{abstract}
We derive  a sufficient condition guaranteeing that a singularly perturbed linear time-varying 
 system is strongly monotone with respect to a matrix cone $C$ of rank~$k$.  
 This implies that the singularly perturbed system inherits the asymptotic properties of systems 
 that are strongly monotone with respect to $C$, which include convergence to the set of 
 equilibria when~$k=1$, and a Poincar\'e-Bendixson property when~$k=2$. We extend this result to
 singularly perturbed nonlinear systems with  a compact and convex state-space. 
 We  
 demonstrate our 
 theoretical results 
 using   a simple
 numerical example.
\end{abstract}
\section{Introduction}\label{sec1}

Dynamical systems whose flow  preserves the ordering induced
by a proper\footnote{A cone is called  proper if it is closed, convex,  and pointed.}   cone~$K$ are called monotone systems. Such systems play an important role in systems and control theory. We begin by reviewing the  classic example that is based on the proper cone~$K=\R^n_+:=\{x\in\R^n : x_i\geq 0,\; i=1,\dots,n\}$, that is, the non-negative  orthant in~$\R^n$. This  cone induces a (partial) ordering~$\leq$ between vectors~$a,b\in\R^n$ defined  by
\[
a\leq b \iff 
b-a\in K
\iff 
a_i \leq b_i ,\; i=1,\dots,n.
\]
Systems whose flow preserves this ordering are called 
positive systems~\cite{farina} 
or cooperative systems~\cite{hlsmith}
and have a rich theory. In particular,    the seminal work of Hirsch~\cite{H84,H88,H82,H85}  (see also the work of Matano in~\cite{matano1979,matano1986}) has led to  
what is now known as  Hirsch's   generic convergence theorem, asserting
that the generic precompact orbit  of a strongly monotone  system approaches the set of equilibria (see Pol\'{a}\v{c}ik \cite{P89,P90} and Smith-Thieme \cite{ST91} for an  improved version of this result).  
  Angeli and Sontag~\cite{monotne_control_systems} have  extended the notion of a monotone system to a system with inputs and outputs, and showed how monotonicity can be applied to analyze  the   feedback connection of such systems. 

  A generalization of monotone systems   are 
  systems whose flow preserves the ordering induced by a cone of rank~$k$ (abbreviated $k$-cone). 
  A set~$C\subset \mathbb{R}^{n}$ is called a $k$-cone if:
\begin{enumerate}
    \item $C$ is closed,
    \item $x\in C \implies \alpha x \in C $ for all $\alpha \in \R$,
    \item $C$ contains a linear $k$-dimensional subspace, and no linear subspace  of  a higher dimension.
\end{enumerate}
The notion of a $k$-cone was introduced by Fusco and Oliva~\cite{FO88,FO91}
in the finite-dimensional case, and by Krasnosel\'skii et al.~\cite{KLS89} in Banach spaces. 
To avoid degenerate cases, we will assume also that~$ C $ has a non-empty  interior.

For example, if~$P\in\R^{n\times n}$ is symmetric with~$k$ negative eigenvalues 
 and~$n-k$ positive eigenvalues then the set
\be\label{eq:quadcone}
C_P^-:= \{x\in\R^n : x^\T Px \leq 0 \}
\ee
is a~$k$-cone. 
 Since~$C_P^-$ is defined by the matrix~$P$, we refer to it as a \emph{matrix $k$-cone}.  
Similarly, 
\[  
C_P^+:= \{x\in\R^n : x^\T Px \geq 0 \}
\]
is a matrix~$(n-k)$-cone. 

In general,  a  $k$-cone is quite different from a proper cone, as it includes a linear subspace  and  it is typically not convex.
Given a~$k$-cone $C$, it is possible to define the  relations:
\[
a \leq ^{\scriptstyle C} b \iff b-a\in C,
\]
and
\[
a \ll ^{\scriptstyle C}   b \iff b-a\in\Int(C),
\]
but these are not (partial) orderings. In fact, they are    neither antisymmetric nor transitive. 
A dynamical system is called monotone with respect to~(w.r.t.)
$C$ if  its flow~$\Phi$ satisfies
\[
a \leq ^{\scriptstyle C}  b \implies \Phi_t(a) \leq ^{\scriptstyle C} \Phi_t(b) \text{ for all } t\geq 0,
\]
and strongly monotone w.r.t. $C$ if, in addition, 
\[
a \leq ^{\scriptstyle C}  b , \;a\not = b\implies \Phi_t(a) \ll^{\scriptstyle C} \Phi_t(b) \text{ for all } t> 0.
\]

For example, consider the LTI system~$\dot x=Ax$, with~$A\in\R^{n\times n}$,  and the matrix $k$-cone  in~\eqref{eq:quadcone}. Fix~$a,b\in\R^n$ such that~$a\leq^{\scriptstyle C_P^-} b$, that is, $(b-a)^\T P(b-a)\leq 0$. 
Then  
\begin{align*}
   & (\Phi_t(b)-\Phi_t(a))^\T P
    (\Phi_t(b)-\Phi_t(a))
   \\& = (\exp(At)(b-a ) ) ^\T P  (\exp(At)(b-a) )
   \\&= 
     (b-a)^\T\left (P+( PA+A^\T P) t +o(t) \right )(b-a)  ,
\end{align*}
implying that a sufficient condition for the  flow to be  strongly monotone w.r.t.  
$C_P^-$ is that  
\be\label{eq:smith_cond}
  PA+A^\T P \preceq -2 s P - \sigma  I_n, \text{ for some } s\geq 0 , \sigma >0.
\ee
This linear matrix inequality~(LMI) resembles the Lyapunov equation, but here $P$ is not necessarily positive-definite. 

Moreover, Eq.~\eqref{eq:smith_cond} implies that 
the function~$V:\R^n\to \R$ defined by~$V(z):= z ^\T P z$ satisfies  
\[
\dot V(x(t)) \leq - 2s V(x(t))-\sigma x^\T(t)x(t) 
\]
  along  a solution of~$\dot x=Ax $.

An important observation is that if~$K$ is 
a proper   cone 
then~$C:=K\cup(-K)$ is a
$1$-cone~\cite{FO91}. Since cooperative  systems admit~$\R^n_+ \cup (-\R^n_+)$ as an invariant set, 
the class of 
systems that are monotone  w.r.t. a $k$-cone 
naturally includes the classical cooperative (and, more generally, monotone) dynamical systems.
However, as noted above there is an essential difference between a $k$-cone, with~$k\geq 2$, and a proper cone, and this requires substantially new ideas for exploring the
implications of  monotonicity  w.r.t. a $k$-cone.

As noted by Sanchez~\cite{Sanchez2009}, there are many important 
examples of flows that are monotone with respect to a~$k$-cone, with~$k\geq 2$, 
including: high-dimensional competitive systems (see, e.g., \cite{B13,B19,H88n,JMW09,M20,M94,MNR19,JN17,WJ01}); systems with quadratic cones like in~\eqref{eq:quadcone}   
and associated Lyapunov-like functions (see, for example, \cite{OS00,R80,R87,2cone_in_SEIRS}); and monotone cyclic feedback systems with negative feedback~\cite{MN13,MS96b,MH90,CTPDS}, arising from a wide range of neural and physiological control systems. A recent example of an important system 
that is monotone with respect to a 2-cone is the antithetic integral feedback system~\cite{margaliot2019compact}. 

To the best of our knowledge,
condition~\eqref{eq:smith_cond} appeared for the first time in the work of  R. A. Smith~\cite{smith_1979,smith1980}. He showed that  strong monotonicity w.r.t. to a matrix 2-cone   implies  a Poincar\'e-Bendixson property: the $\omega$-limit set of a bounded orbit containing no equilibria is a closed orbit. 
Such a result can be useful in establishing that a dynamical system admits a closed orbit, which is in general not trivial.

Consider a system that is  
  strongly monotone with respect to a (not necessarily matrix) $k$-cone   $C$, with~$k\geq2$.
A nontrivial orbit~$ O(a) :=\{\Phi_t(a): t\geq 0\}$ is called \emph{pseudo-ordered} if  
 there exist times~$\tau_1 \not= \tau_2$ such that~$\Phi_{\tau_1}(a)   \leq ^{\scriptstyle C}  \Phi_{\tau_2}(a)$. Otherwise, it is called \emph{unordered}.
The structure of an $\omega$-limit set of a pseudo-ordered  orbit
is much more complicated than in the case of a system that is monotone w.r.t. a proper cone, due to the lack of convexity of~$C$. Sanchez~\cite{Sanchez2009,sanchez2}   showed that the closure of any orbit in the $\omega$-limit set of a pseudo-ordered orbit is ordered with respect to~$C$, and  proved  a Poincar\'{e}-Bendixson property
  for pseudo-ordered orbits when~$k=2$.  
Sanchez considered  smooth finite dimensional flows, and used  the $C^1$-Closing Lemma to prove his results.
Feng, Wang and Wu~\cite{feng2017} considered the more general case of  a semiflow on a Banach space,  without the smoothness requirement, and  proved  that the $\omega$-limit set $\omega(x)$  of a pseudo-ordered semiorbit admits a trichotomy: $\omega(x)$ is either ordered; or $\omega(x)$ is contained in the set of equilibria; or $\omega(x)$ possesses a certain ordered homoclinic property.  

 Recently, Feng,   Wang, and  Wu~\cite{FWW19} 
 proved that a system that is
 strongly monotone 
 w.r.t. a~$k$-cone
 admits an open and dense subset~$D$ of the phase space
 such that 
 orbits with initial point in~$D$ 
 are either pseudo-ordered or convergent to equilibria. This covers the celebrated Hirsch’s Generic Convergence Theorem for 
 $k=1$, and yields a generic Poincar\'{e}-Bendixson property for $k=2$. Their  approach uses  
 ergodic arguments based on  $k$-exponential separation.

Wang,  Yao and  Zhang~\cite{prevalent_wang}
also analyzed systems that are strongly monotone w.r.t. a $k$-cone from a 
measure-theoretic perspective.
They showed that prevalent (or,  equivalently, almost all) orbits will be pseudo-ordered or convergent to equilibria. This reduces to Hirsch’s prevalent convergence Theorem for~$k=1$; and implies an almost-sure Poincar\'{e}–Bendixson Theorem in  the case $k=2$.  

 Weiss and Margaliot~\cite{Eyal_k_posi}  introduced   $k$-cooperative  systems, i.e.,   ODE systems that are strongly monotone with respect to the $k$-cone of vectors in~$\R^n$ with up to~$k-1$ sign variations.  They showed that 
  $1$-cooperative systems are just cooperative systems, 
  cyclic feedback   systems with negative feedback are~$2$-cooperative systems, 
	and~competitive systems are (up to a coordinate transformation)~$(n-1)$-cooperative systems. 
  They also    established  a 
strong Poincar\'e-Bendixson property 
for strongly $2$-cooperative nonlinear systems, namely, if the $\o$-limit set of a  precompact solution
  does not include an equilibrium point then it is a closed orbit.
As noted above,  Sanchez~\cite{Sanchez2009} proved  a Poincar\'e-Bendixson property for  \emph{certain} trajectories
of a system that is strongly monotone with respect to a $2$-cone. The results for~2-cooperative systems
are stronger: they  apply to every precompact trajectory.
See~\cite{margaliot2019compact} for an
  application of these results to a   closed-loop model  
from systems biology.

Forni, Sepulchre and their colleagues~\cite{Forni2018,Miranda-Villatoro2018}  
considered   monotone systems w.r.t. a matrix $k$-cone
(see~\eqref{eq:quadcone}) in a systems and control perspective, calling them~$k$-dominant systems, and used~\eqref{eq:smith_cond} as the definition of $k$-dominance with rate~$s$. Importantly, using the fact that~\eqref{eq:smith_cond} is
 an LMI, they extended the notion of a system that is monotone w.r.t. a matrix $k$-cone to systems with inputs and outputs, and studied the interconnections of such systems.  

Singularly perturbed systems \cite{Kokotovic1999}, \cite[Ch.~11]{Khalil2002} 
appear naturally in many control applications. For instance, see~\cite{LorenzettiTAC,LorenzettiAutomatica} for an application to low-gain integral control, and~\cite{Lorenzetti2022} for an application to power systems stability.

Wang and Sontag~\cite{wang_sontag_sing}
showed that  a singularly perturbed strongly monotone
system inherits the generic convergence properties of strongly monotone
systems, and demonstrated an application in systems biology.  Niu and Xie~\cite{sing_2cone} showed that singularly perturbed systems that are strongly monotone w.r.t. a 2-cone   inherit the 
generic Poincar\'{e}-Bendixon property. The results 
in~\cite{wang_sontag_sing,sing_2cone} are based on geometric singular perturbations theory~\cite{fenichel_geo,Sakamoto1990}. 

Here, we   consider $k$-dominance for singularly perturbed systems. We derive  a new sufficient condition  for  a singularly perturbed system to be  $k$-dominant. Our proof uses standard tools from linear algebra and dynamical systems theory, and does not require the heavy machinery of geometric singular perturbations theory. Another advantage of our approach is that it  provides an \emph{explicit} matrix cone~$C_\mathcal{P}^-$ such that the perturbed system is strongly monotone w.r.t. this cone for all~$\e>0$ sufficiently small. 

The paper is organized as follows.  
Section~\ref{sec:pre}
describes some preliminary results that are used later on. Section~\ref{sec:main}
describes our main results, and the final section concludes.

 We use standard notation. 
Vectors [Matrices] are denoted by 
small [capital] letters. 
Denote $\R_+ : =\{x\in\R \ | \ x\geq0\}$. 
A matrix~$A\in\R^{n\times n}$ is called positive-definite (denoted~$A\succ 0)$ if~$A$ is symmetric and~$x^\T Ax>0$ for all~$x\in\R^n\setminus\{0\}$. 
The inertia of a symmetric matrix~$A\in\R^{n\times n}$
 is the triple
\[
\In(A)  : = (\nu(A) ,  \zeta(A), \pi (A) ) ,
\]
in which $\nu (A)$, $\zeta(A)$ and $\pi(A)$ stand for the number of negative, zero,
and positive eigenvalues of~$A$, respectively. For example,~$\In(A)=(0,0,n)$ 
[$\In(A)=(n,0,0)$]
iff~$A$ is positive-definite [negative-definite]. 

\section{Preliminaries}\label{sec:pre}
For the reader's
convenience, we recall 
some background on singular perturbation theory  from \cite[Chap.~5]{Kokotovic1999}.

\subsection{Singular Perturbation Theory for LTV Systems}\label{subsec:SP}

Consider the LTV singularly perturbed system
\begin{align}\label{eq:LTV_SP}
 \dot x(t)&=A(t)x(t)+B(t)z(t),\nonumber\\
  \e \dot{z}(t)&=C(t)x(t)+D(t)z(t) , 
\end{align}
where $x:\R_+ \to \R^{n_r}, z:\R_+ \to \R^{n_f}$, and $\e>0$ small.

\begin{Assumption}\label{ass:1}
There exists $c_1>0$ such that 
$$\mathrm{Re} ( \lambda(D(t)))\leq -c_1<0 \text{ for all }  t\geq0.$$
\end{Assumption}

\begin{Assumption}\label{ass:2}
The matrices $A(t),B(t),C(t),D(t)$ are continuously differentiable,
bounded, and with bounded derivatives for all $t\geq0$.
\end{Assumption}

\begin{Theorem}
\cite[Theorem~3.1 (Ch.~5)]{Kokotovic1999}
\label{thm:Kok_LH}
Assume that the matrices $A,B,C,D$ satisfy Assumptions~\ref{ass:1} and~\ref{ass:2}.
Then, there exists an $\e^*>0$ such that for all $\e\in(0,\e^*)$ there exist continuously
differentiable matrices $L_\e(t)$ and $H_\e(t)$, bounded for all $t\geq0$, satisfying
\begin{align}\label{eq:Kok_LH}
\e\dot{L}_\e&=DL_\e-C-\e L_\e ( A-BL_\e ) , \\
-\e\dot{H}_\e&=H_\e D-B+\e H_\e L_\e B-\e ( A-BL_\e ) H_\e.
\end{align}
Moreover, 
\begin{align}\label{eq:Llasym}
L_\e=L_0+O(\e),\quad H_\e=H_0+O(\e),
\end{align}
with~$L_0:=D^{-1}C$, $ H_0:=BD^{-1}$.    
\end{Theorem}

Under the assumptions of   Theorem~\ref{thm:Kok_LH}, the 
matrix
\begin{equation}\label{eq:T}
T_\e(t):=\bbm{I_{n_r} & \e H_\e(t) \\ -L_\e(t) & I_{n_f}-\e L_\e(t)H_\e(t)},
\end{equation}
is well-defined and non-singular   for all $t\geq0$.
Introducing the state-variables $\sbm{\xi \\ \eta}:=(T_\e)^{-1}\sbm{x \\ z}$ transforms 
the singularly perturbed LTV system~\eqref{eq:LTV_SP} to  the decoupled system: 
\begin{align}\label{eq:LTV_xieta}
\begin{bmatrix} \dot \xi \\ \dot \eta \end{bmatrix}=\mathcal{A}_\e(t) 
\begin{bmatrix}  \xi \\   \eta \end{bmatrix},
\end{align}
with
\begin{equation}\label{eq:A_e}\mathcal{A}_\e(t) 
:= \bbm{A(t)-B(t)L_\e(t) & 0 \\ 0 & \e^{-1} D(t)+ L_\e(t)B(t)}.
\end{equation}

\begin{Remark}
It can be verified that $\det( T^{-1}_\e(t))=1$
for all~$t\geq0$ and  all~$\e>0$.
Moreover, $T_\e$ is a Lyapunov transformation, namely,
it is stability-preserving (see \cite[Ch.~5]{Kokotovic1999} for more details).
\end{Remark}

\begin{Remark}
Theorem~\ref{thm:Kok_LH} guarantees that the matrices
$L_\e(t)$  [$H_\e(t)$] converge  when~$\e\to0$  to
$L_0$ [$H_0$] (uniformly in~$t\geq0$), i.e., to the
solutions of \eqref{eq:Kok_LH} obtained by
setting~$\e=0$.
\end{Remark}

The reduced (slow) model associated to \eqref{eq:LTV_SP} is
\begin{equation}\label{eq:LTV_red}
\dot{x}_s(t)=A_0(t)x_s(t), \text{ with } A_0(t) : =A(t)-B(t)L_0(t),
\end{equation}
while the boundary-layer (fast) systems is given by
\begin{equation}\label{eq:LTV_fast}
\e\dot{z}_f(t)=D(t)z_f(t),
\end{equation}
where here~$t\geq0$ is treated as a fixed parameter.

Singular perturbation theory allows to deduce the stability of~\eqref{eq:LTV_SP} 
based on properties of the slow and fast systems. The next result demonstrates this. 

\begin{Theorem}\cite[Theorem~4.1 (Ch.~5)]{Kokotovic1999}
\label{thm:Kok_LTV}
Assume that the matrices $A,B,C,D$ in~\eqref{eq:LTV_SP} satisfy Assumptions~\ref{ass:1} and~\ref{ass:2}.
Further, assume that the slow system \eqref{eq:LTV_red} is uniformly asymptotically stable.
Then, there exists an $\e^*>0$ such that for any $\e\in(0,\e^*)$ the singularly 
perturbed system \eqref{eq:LTV_SP} is uniformly asymptotically stable. In particular, the state transition matrix $\phi_{\e}$  of \eqref{eq:LTV_SP}
satisfies 
$$\|\phi_{\e}(t,s)\|\leq m e^{-\alpha(t-s)} \text{ for all }  t\geq s\geq0,$$
for some~$m,\alpha>0$ that are independent of~$\e$.
\end{Theorem}

The proof of  this theorem is based on using   $T_\e$ in \eqref{eq:T}
to transform \eqref{eq:LTV_SP} into \eqref{eq:LTV_xieta}, and  observing that $\xi$ is
a regular perturbation of the uniformly asymptotically stable (slow) system \eqref{eq:LTV_red}
(see Theorem~\ref{thm:Kok_LH}). Similarly, the stability of~$\eta$ is linked to that
of the fast variable $z_f$  in \eqref{eq:LTV_fast}. For a detailed proof, we refer the reader
to \cite[Sec.~5.2]{Kokotovic1999}.

\section{Main results}\label{sec:main}
This section includes our main results. 

\subsection{$p$-Dominance of a Singularly Perturbed LTV System}

Our first result provides  a sufficient condition for $p$-dominance of~\eqref{eq:LTV_SP}.
\begin{Theorem}\label{thm:LTV}
Consider the singularly perturbed LTV system~\eqref{eq:LTV_SP},  
with~$A,B,C,D$ satisfying Assumption~\ref{ass:2}.
 Suppose that there exist symmetric matrices~$P_r\in\R^{n_r\times n_r}$
 with $\In(P_r)=(p,0,n_r-p)$, and $P_f\in\R^{n_f\times n_f}$ with
$\In(P_f)=(0,0,n_f)$, rates $\lambda_r,\lambda_f \geq 0$ and~$\sigma_r,\sigma_f>0 $ such that
\begin{subequations}\label{cond:reduced_dominant}
\begin{equation}\label{eq:p-dom_LTV}
P_r A_0 (t) + A_0^\T(t)  P_r \preceq -2 \lambda_r P_r-\sigma_r I_{n_r}, 
\end{equation}
\begin{equation}\label{cond:fast_dominant_LTV}
P_f D(t) + D^\T(t) P_f \preceq -2 \lambda_f P_f - \sigma_f I_{n_f},
\end{equation}
\end{subequations}
for all $t \ge 0$.
Let
$
\mathcal{P} := \begin{bmatrix}
P_r & 0 \\
0 & P_f
\end{bmatrix}$.
Then  there exists $\e^* > 0$ such that for any~$\e \in (0,\e^*)$, system~\eqref{eq:LTV_SP} is $p$-dominant w.r.t.~$\mathcal{P}$ with rate $\lambda_r$.
\end{Theorem}

Note that
\eqref{eq:p-dom_LTV}
 is a $p$-dominance condition for the reduced model~\eqref{eq:LTV_red}.
Since~$P_f$ is positive-definite, \eqref{cond:fast_dominant_LTV} is a contraction condition on the fast system~\eqref{eq:LTV_fast}
w.r.t. to a scaled~$L_2$ norm (see, e.g.,~\cite{LOHMILLER1998683,sontag_cotraction_tutorial}). This latter condition
is a ``strong version'' of Assumption~\ref{ass:1}.

\begin{proof}
 Eq.~\eqref{cond:fast_dominant_LTV}  implies that Assumption~\ref{ass:1} holds, so
 by 
     Theorem~\ref{thm:Kok_LH}  it is enough to analyze~\eqref{eq:LTV_xieta}. 
     Let~$\sigma :=(1/2)\min\{ \sigma_r,\sigma_f\}>0$.
     We  claim that there exists~$\e^*>0$ 
     such that for any~$\e\in(0,\e^*)$, we have 
     \begin{align*}
         \mathcal{P}\mathcal{A}_\e(t)+\mathcal{A}^\T_\e(t)\mathcal{P} \preceq -2\lambda_m \mathcal{P} -  \sigma  I_{n_r+n_f}, 
     \end{align*}
     for all~$t\geq0$.  To simplify the notation, we omit the dependence on~$t$ from now on. By the block-diagonal form of~$\mathcal{A}_\e$, it is enough to show that for any~$t\geq 0$, we have
     \begin{align}\label{eq:fwn1}
            P_r (A -B L_{\e})+(A -B L_{\e})^\T  P_r \preceq -2\lambda_r P_r- \sigma  I_{n_f},  
     \end{align}
     and
     \begin{multline}\label{eq:fwn2}
            P_f (\e^{-1} D+L_{\e}B )+(\e^{-1} D+L_{\e}B )^\T  P_f \\ \preceq -2\lambda_r P_f-  \sigma  I_{n_f},  
     \end{multline}
Combining~\eqref{eq:p-dom_LTV}, Assumption~\ref{ass:2}
and~\eqref{eq:Llasym}
implies that~\eqref{eq:fwn1} holds. Furthermore, since $L_\e$ and $B$ are bounded, it follows from~\eqref{cond:fast_dominant_LTV} that the left hand side of~\eqref{eq:fwn2} can be made arbitrarily negative-definite  by taking $\e>0$ small enough, and in particular such that~\eqref{eq:fwn2} holds, and this completes the proof.
\end{proof}

\vspace{2mm}

Specializing   Theorem~\ref{thm:LTV} to the case of  LTI systems yields the following result.

\begin{Corollary}\label{coro:lti_case}
Consider the LTI singularly perturbed system:
\begin{align}\label{eq:LTI_SP}
 \dot x(t)&=A x(t)+B z(t),\nonumber\\
  \e \dot{z}(t)&=C x(t)+D z(t) , 
\end{align}
with~$D $   invertible, and let~$  A_0 : =  A-BD^{-1}C $. 
Suppose that there exist symmetric matrices~$P_r, P_f\in\R^{n\times n}$ 
with  $\In(P_r)=(p,0,n_r-p)$,
$\In(P_f)=(0,0,n_f)$,   rates~$\lambda_r,\lambda_f\geq 0$, and~$\sigma_r,\sigma_f>0$  such that
\begin{subequations} 
\begin{equation} 
P_r  A_0   +   A_0^\T   P_r \preceq -2 \lambda_r P_r -\sigma_r I_{n_r}, 
\end{equation}
\begin{equation} 
P_f D  + D^\T  P_f \preceq -2 \lambda_f P_f-\sigma_f I_{n_f}. 
\end{equation}
\end{subequations}
Then  there exists $\e^* > 0$ such that for any~$\e \in (0,\e^*)$, system~\eqref{eq:LTI_SP} is $p$-dominant with respect to~$\mathcal{P}$ with rate~$\lambda_r$. 
\end{Corollary}
In the particular case~$p=0$ (so~$P_r\succ 0$), this implies that for any~$\e>0$ sufficiently small the perturbed system is~$0$-dominant  w.r.t. to the positive-definite matrix~$\mathcal{P}$, and thus it is asymptotically stable. This recovers Theorem~\ref{thm:Kok_LTV}. 

\subsection{$p$-Dominance in a Singularly Perturbed Nonlinear System}

We now extend   the analysis to time-invariant nonlinear systems.  The basic idea is to consider the variational equation
associated with the singularly perturbed nonlinear system (this is sometimes referred to as \emph{differential analysis} or \emph{incremental
analysis}  
\cite{Forni2018,sontag_cotraction_tutorial,kordercont}). In our case, the variational system is a singularly perturbed  LTV system. 

Consider the singularly perturbed nonlinear system
\begin{align} \label{eq:cl_original}
\dot{x}&=f(x,z), \nonumber \\
\e\dot{z}&=g(x,z),
\end{align}
where $\e>0$, $f\in C^2(\mathbb{R}^{n_r}\times
\mathbb{R}^{n_f};\mathbb{R}^{n_r})$, $g\in C^2(\mathbb{R}^{n_r}\times\mathbb{R}^{n_f};\mathbb{R}^{n_f})$, and~$f(
0,0)=0$, $g(0,0)=0$. 

We assume that~\eqref{eq:cl_original} admits a
compact and convex state-space~$\Omega_x\times \Omega_z \subset \R^{n_r}\times\R^{n_f}$ for any~$\e>0$.
Given an initial condition~$\begin{bmatrix}x_0 \\ z_0\end{bmatrix}\in\Omega_x \times \Omega_z$, 
the    variational system of~\eqref{eq:cl_original} is  
\begin{equation}\label{eq:cl_incremental}
\begin{gathered}
\dot{x}_\d=\frac{\partial f}{\partial x}(x,z)x_\d+\frac{\partial f}{\partial z}(x,z)z_\d, \\
\e\dot{z}_\d=\frac{\partial g}{\partial x}(x,z)x_\d+\frac{\partial g}{\partial z}(x,z)z_\d,
\end{gathered}
\end{equation}
where the Jacobians are evaluated  along the trajectory~$\sbm{x(t) \\ z(t)} \in\Omega_x \times \Omega_z$
of \eqref{eq:cl_original}. Note that~\eqref{eq:cl_incremental} is a singularly perturbed LTV system. 

We assume that there exists~$h\in C^2(\mathbb{R}^{n_r};\mathbb{R}^{n_f})$
such that~$g(x,h(x))=0$ for all $x\in\mathbb{R}^{n_r}$.
Then, the reduced model associated with \eqref{eq:cl_incremental} is  
\begin{equation}\label{eq:red_incremental}
\dot{x}_s= \left[\frac{\partial f}{\partial x}(x,z)+\frac{\partial f}{\partial z}(x,z)\frac{\partial h}{\partial x}(x)\right]x_s,
\end{equation}
where 
\begin{equation}\label{eq:h_d}
\frac{\partial h}{\partial x}(x)=-\left[\frac{\partial g}{\partial z}(x,h(x))\right]^{-1}\frac{\partial g}{\partial x}(x,h(x)).
\end{equation}
The boundary-layer associated to \eqref{eq:cl_incremental} is  
\begin{equation}\label{eq:fast_incremental}
\e \dot{z}_f=\frac{\partial g}{\partial z}(x,z_f+h(x))z_f,
\end{equation}
where here~$x\in\mathbb{R}^{n_r}$ is treated as a fixed parameter. Let
\begin{equation}\label{eq:NLmatrices}
\begin{gathered}
A:=\frac{\partial f}{\partial x}(x,z), \quad B:=\frac{\partial f}{\partial z}(x,z), \\ 
C:=\frac{\partial g}{\partial x}(x,z), \quad D:=\frac{\partial g}{\partial z}(x,z),
\end{gathered}
\end{equation}
where for simplicity  we  omit
the dependence of $A,B,C,D$ on $(x(t),z(t))$. Let
\begin{equation}\label{eq:NLL0}
\begin{gathered}
L_0:=D^{-1}C \quad \text{and} \quad A_0:=A-BL_0.
\end{gathered}
\end{equation}

\begin{Assumption}\label{ass:NL}
The matrices $A,B,C,D$
are continuously differentiable, bounded, and with bounded derivatives for any initial condition
$(x_0,z_0)\in\Omega_x\times\Omega_z$ of~\eqref{eq:cl_original}.
\end{Assumption}

We can now state our second main result.
\begin{Theorem}\label{thm:NL}
Consider the singularly perturbed nonlinear system~\eqref{eq:cl_original}, and assume that 
  Assumption~\ref{ass:NL} holds. Suppose that there exist symmetric matrices~$P_r\in\R^{n_r\times n_r}$
 with $\In(P_r)=(p,0,n_r-p)$, and $P_f\in\R^{n_f\times n_f}$ with
$\In(P_f)=(0,0,n_f)$, rates $\lambda_r,\lambda_f \geq 0$ and~$\sigma_r,\sigma_f>0 $ such that
\begin{subequations}\label{cond:nonlin_reduced_dominant}
\begin{equation}\label{eq:p-dom_vari}
P_r A_0 + (A_0)^\T P_r \preceq -2 \lambda_r P_r-\sigma_r I_{n_r}, 
\end{equation}
\begin{equation}\label{cond:fast_dominant_vari}
P_f D + D^\T P_f \preceq -2 \lambda_f P_f - \sigma_f I_{n_f},
\end{equation}
\end{subequations}
for all $(x,z)$ emanating from initial conditions $(x_0,z_0)\in\Omega_x\times\Omega_z$ of \eqref{eq:cl_original}.
Let
$$
\mathcal{P} := \begin{bmatrix}
P_r & 0 \\
0 & P_f
\end{bmatrix}.$$
Then  there exists $\e^* > 0$ such that for any~$\e \in (0,\e^*)$, the system~\eqref{eq:cl_original} is $p$-dominant w.r.t.~$\mathcal{P}$ with rate $\lambda_r$.
\end{Theorem}

 \begin{proof}
Fix~$(x_0,z_0),(\tilde{x}_0,\tilde{z}_0) \in\Omega_x\times \Omega_z$, and let~$(x(t),z(t)),(\tilde x(t),\tilde z(t))$ denote the corresponding solutions of~\eqref{eq:cl_original}. 
Let~$v(t):=x(t)-\tilde x(t)$, $w(t):=z(t)-\tilde z(t)$, $\bar x(t,r):=rx(t)+(1-r)\tilde x(t)$,
and~$\bar z(t,r):=rz(t)+(1-r)\tilde z(t)$, with~$r\in[0,1]$. Then
\begin{align}\label{eq:vwys}
    \dot v&=\int_0^1 \frac{\partial f}{\partial x}(\bar x,\bar z)\diff r\; v +\int_0^ 1 \frac{\partial f}{\partial z}(\bar x,\bar z)\diff r\;  w,  \\
      \e \dot w&=\int_0^1 \frac{\partial g}{\partial x}(\bar x,\bar z)\diff r\; v +\int_0^ 1 \frac{\partial g}{\partial z}(\bar x,\bar z)\diff r\;  w \nonumber .
\end{align}
By Theorem~\ref{thm:LTV}, conditions~\eqref{cond:nonlin_reduced_dominant}
imply that  
\eqref{eq:vwys}  is~$p$-dominant for all~$\e>0$ sufficiently  small. Using the compactness of~$\Omega_x\times \Omega_z$    implies that there exists~$\e^*>0$ such that~\eqref{eq:vwys} is~$p$-dominant for all~$\e \in (0,\e^*)$ 
and all~$(x_0,z_0)\in \Omega_x\times \Omega_z$.
In other words,  
\[
\begin{bmatrix}
    x_0\\z_0
\end{bmatrix}  - 
\begin{bmatrix}
   \tilde x_0\\\tilde z_0
\end{bmatrix} \in C^-_\mathcal{P}\setminus\{0\} \implies
\begin{bmatrix}
    x(t)\\z(t)
\end{bmatrix}  - 
\begin{bmatrix}
   \tilde x(t)\\\tilde z(t)
\end{bmatrix} \in \Int(C^-_\mathcal{P})
\]
for all~$t >  0$, 
and this completes the proof.
\end{proof}
 
The next example  demonstrates Theorem~\ref{thm:NL}
  using a singularly perturbed version of a mechanical system with a  nonlinear
spring  from~\cite{Forni2018}.

\begin{Example}
Consider the non-linear singularly perturbed system:
\begin{align}\label{eq:non_example}
    \dot x_1& = x_2, \nonumber  \\
    \dot x_2&= v(x_1)-5 z ,\nonumber \\
    \e \dot z & = x_2-z,  
\end{align}
where~$v:\R\to\R$ is~$C^2$, $v(0)=0$, and~$ v' (s) \in [-5 , 2 ]$ for all~$s\in \R$. The associated variational system is the LTV:
\begin{align}
    \dot {\delta x}_1& =\delta  x_2, \nonumber  \\
    \dot {\delta x }_2&= v'(x_1) \delta x_1 -5 \delta z ,\nonumber \\
    \e \dot {  \delta z} & =\delta  x_2- \delta z. 
\end{align}
We can write this as the LTV system \eqref{eq:LTV_SP} with matrices
$$A =\begin{bmatrix} 0 & 1 \\  v'(x_1) &  0 \end{bmatrix}, \ B =\begin{bmatrix} 0 \\ -5\end{bmatrix},
\ C=\begin{bmatrix}  0& 1 \end{bmatrix}, \ D=-1.$$ 
Following Sec.~\ref{sec:pre}, we define
$$L_0:=D^{-1}C= -C, \quad A_0:=A-BL_0=\begin{bmatrix} 0 & 1 \\  v'(x_1) &  -5 \end{bmatrix}.$$ 
Clearly, $A_0$ is always in the convex hull of the matrices
$$ \underline{M} :=\begin{bmatrix}
0&1\\-5 & -5 
\end{bmatrix}, \quad \overline{M}:=\begin{bmatrix}
0&1\\ 2 & -5 
\end{bmatrix}.$$ 
We choose $$P_r:=\begin{bmatrix}
-5.1987 & 3.6260\\ 3.6260 &  6.1987
\end{bmatrix}. $$ Then~$\In(P_r)=(1,0,1)$, and
\begin{align*}
P_r  \underline{M}+ \underline{M}^\T P_r&\preceq -2\lambda_r P_r -\sigma_r I_2 ,  \\
P_r \overline{M}+ \overline{M}^\T P_r &\preceq -2\lambda_r P_r -\sigma_r  I_2, 
\end{align*}
for~$\lambda_r =2$,~$\sigma_r=0.01$, which corresponds to \eqref{eq:p-dom_vari}. Similarly, 
for~$P_f=1$, $\lambda_f=1/2$, and~$\sigma_f = 1$, we have
$$ P_f D +  D^\T P_f \preceq -2 \lambda_f P_f -\sigma_f, $$
which corresponds to \eqref{cond:fast_dominant_vari}. We conclude that all the conditions in Theorem~\ref{thm:NL}  hold, so  the system is~$1$-dominant. Thus, every bounded trajectory converges to an equilibrium point. 
Fig.~\ref{fig:simu} depicts several  trajectories  of~\eqref{eq:non_example} with~$v(x)=7\tanh(x)-5x$ and~$\e=0.01$. It may be seen that indeed  every trajectory converges to an equilibrium point. 
\end{Example}

\begin{figure}
    \centering
    \includegraphics[width=\linewidth]{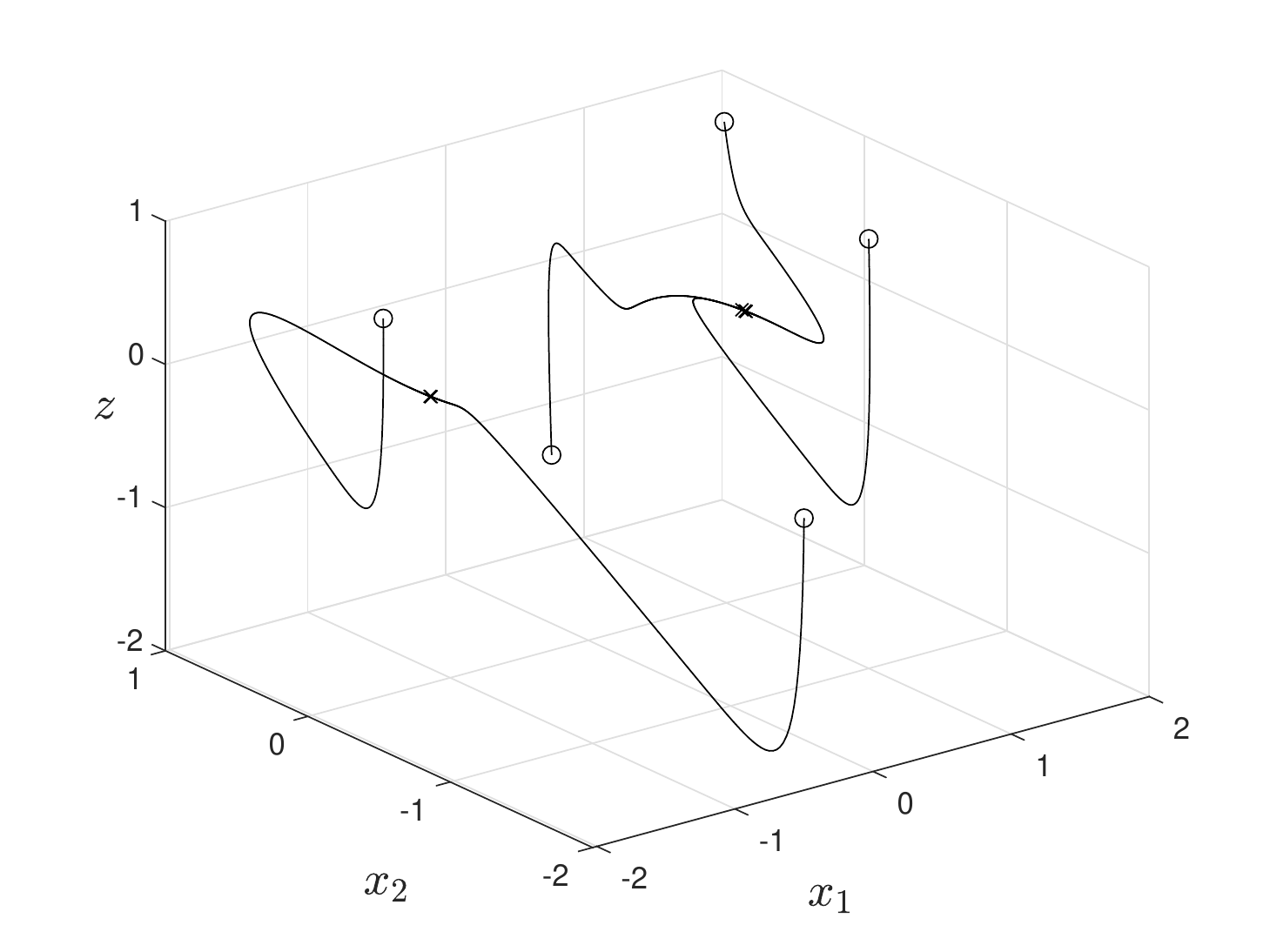}
    \caption{Several trajectories of~\eqref{eq:non_example}. Initial conditions [equilibrium points] are marked with '$\circ$' ['x'].  }
    \label{fig:simu}
\end{figure}

\section{Conclusion}
Systems that are monotone w.r.t. a   cone of rank~$k$ satisfy  useful asymptotic properties. This follows from the fact that, roughly speaking, almost all bounded solutions of such  a system  can be injectively projected into a $k$-dimensional subspace, so   essentially the dynamics is $k$-dimensional~\cite{Sanchez2009}. 

We   derived  a  sufficient condition guaranteeing that a singularly perturbed nonlinear system is monotone w.r.t. a matrix cone of rank~$k$. The   analysis is relatively simple and does  not require using geometric singular perturbations theory. Another advantage of our approach is that it  provides an explicit expression for the matrix $k$-cone~$C_{\mathcal P}^-$. 

We demonstrated our results using a synthetic example, but we believe that 
it may be useful in the analysis of real-world systems, e.g., power systems.
Another possible direction for future research is to study the relation of our results to singularly perturbed $k$-contractive systems. 

\subsection*{Acknowledgments}
We are grateful to Yi Wang for helpful comments. 

\bibliographystyle{IEEEtran}

\end{document}